\documentclass[10pt]{amsart}

\usepackage{ amssymb,stmaryrd,textcomp}

\theoremstyle{plain}
\newtheorem{thm}{Theorem}[section]

\newtheorem{cor}[thm]{Corollary}
\newtheorem{lem}[thm]{Lemma}
\newtheorem{prop}[thm]{Proposition}

\theoremstyle{definition}
\newtheorem{defi}[thm]{Definition}
\theoremstyle{remark}
\newtheorem{rem}[thm]{Remark}
\numberwithin{equation}{section}

\newtheorem{example}[thm]{Example}

\newcommand{\lgw}{\longrightarrow}
\newcommand{\lgm}{\longmapsto}

\newcommand{\lb}{\llbracket}
\newcommand{\rb}{\rrbracket}

\newcommand{\Ker}{\text{Ker}}

\newcommand{\wdh}{\widehat}

\newcommand{\NN}{\mathcal N}

\newcommand{\wdt}{\widetilde}

\renewcommand{\k}{\Bbbk}

\newcommand{\F}{\mathcal F}

\newcommand{\N}{\mathbb{N}}

\newcommand{\C}{\mathbb{C}}

\newcommand{\lag}{\langle}
\newcommand{\rag}{\rangle}

\renewcommand{\a}{\alpha}
\renewcommand{\b}{\beta}

\newcommand{\s}{\sigma}
\newcommand{\m}{\mathfrak m}

\renewcommand{\phi}{\varphi}

\renewcommand{\o}{\omega}

\makeatletter
\newsavebox{\@brx}
\newcommand{\llangle}[1][]{\savebox{\@brx}{\(\m@th{#1\langle}\)}%
  \mathopen{\copy\@brx\mkern2mu\kern-0.9\wd\@brx\usebox{\@brx}}}
\newcommand{\rrangle}[1][]{\savebox{\@brx}{\(\m@th{#1\rangle}\)}%
  \mathclose{\copy\@brx\mkern2mu\kern-0.9\wd\@brx\usebox{\@brx}}}
\makeatother

\renewcommand{\lg}{\llangle}
\newcommand{\rg}{\rrangle}

\begin{document}
\title[Linear nested Artin Approximation]{Linear nested Artin Approximation Theorem for algebraic power series}
\author[F.~J. Castro-Jim\'enez,  D. Popescu, G. Rond]{Francisco-Jes\'us Castro-Jim\'enez,  Dorin Popescu, Guillaume Rond}

\address{Departamento de \'Algebra, Universidad de Sevilla, Spain}
\email{castro@us.es}

\address{Simion Stoilow Institute of Mathematics of the Romanian Academy, Research unit 5,
University of Bucharest, P.O.Box 1-764, Bucharest 014700, Romania}
\email{dorin.popescu@imar.ro}

\address{Aix-Marseille Universit\'e, CNRS, Centrale Marseille, I2M, UMR 7373, 13453 Marseille, France}
\email{guillaume.rond@univ-amu.fr}

\begin{abstract} We give an elementary proof of the nested Artin approximation Theorem for linear equations with algebraic power series coefficients.
Moreover, for any Noetherian local subring of the ring of formal power series, we clarify the relationship between this theorem and the problem of the commutation of two operations for ideals:  the operation of replacing an ideal by its completion and the operation of replacing an ideal by one of its elimination ideals. In particular we prove that a Grothendieck conjecture about morphisms of analytic/formal algebras and Artin's question about linear nested approximation problem are equivalent.
\end{abstract}


\thanks{F.-J.~Castro-Jim\'enez was partially supported by Ministerio de Econom\'\i a y Competitividad (MTM2013-40455-P and Feder). D. Popescu was partially  supported by the project  ID-PCE-2011-1023, granted by the Romanian National Authority for Scientific Research, CNCS - UEFISCDI.
G. Rond was partially supported by ANR projects STAAVF (ANR-2011 BS01 009) and SUSI (ANR-12-JS01-0002-01)}

\maketitle
\section{Introduction}
The aim of the paper is to investigate the nested Artin  approximation problem for linear equations. Namely the  nested Artin approximation problem is the following: if
$$F(x,y)=0$$ is a system of algebraic or analytic equations which are linear in $y$, with $x=(x_1,\ldots,x_n)$ and $y=(y_1,\ldots,y_m)$, and if $y(x)$ is a formal power series solution $$F(x,y(x))=0$$ with the property that
\begin{equation}\label{nested_prop} y_i(x) \text{ depends only on the variables } x_1,\ldots, x_{\s_i} \end{equation}
for some integers $\s_i$, is it possible to find algebraic or analytic solutions satisfying \eqref{nested_prop}?\\

D. Popescu \cite{Po} proved the nested Artin approximation theorem for any vector $F(x,y)$ of algebraic power series (not necessarily linear in $y$). In this paper we give an elementary proof of this theorem when $F(x,y)$ is linear in $y$.  Moreover we provide a characterization for a certain class of germs of functions $F(x,y)$, linear  in $y$, to satisfy the nested Artin approximation property. From an example of A. Gabrielov \cite{Ga} we know that the answer to the nested Artin  approximation problem is negative for the ring of convergent power series.

In order to explain the situation let us consider the following theorem (proven by M. Artin in characteristic zero and by M. Andr\'e in positive characteristic):\\

\begin{thm}\label{Ar}\cite{Ar68}\cite{An}
Let $\k$ be a complete valued field and let $F(x,y)$ be a vector of convergent power series in two sets of variables $x$ and $y$. Assume given a formal power series solution $\wdh{y}(x)$ vanishing at $0$, $$F(x,\wdh{y}(x))=0.$$
Then,  for any $c\in\N$,  there exists a  convergent power series  solution $\wdt{y}(x)$,
$$F(x,\wdt{y}(x))=0$$ which coincides with $\wdh{y}(x)$ up to degree $c$, $$\wdt{y}(x)\equiv \wdh{y}(x) \text{ modulo } (x)^c.$$
\end{thm}

Then M. Artin  (see \cite[p.7]{Ar71}) asked, whether or not, given a formal solution $\wdh y(x)=(\wdh y_1(x), \ldots,\wdh y_m(x))$ satisfying
$$\wdh y_j(x)\in \k\lb x_1, \ldots, x_{\s_j}\rb\ \ \ \forall j$$
for some integers $\s_j\in \{1, \ldots,n\}$, there exists a convergent solution $\wdt y(x)$ as in Theorem 1.1 
such that $$\wdt y_j(x)\in\k\{x_1, \ldots,x_{\s_j}\} \ \ \forall j.$$

Shortly after,  A. Gabrielov \cite{Ga} gave an example showing that the answer to Artin's question is negative in general. \\
On the other hand since Theorem \ref{Ar} remains valid if we replace convergent power series by algebraic power series (cf. \cite{Ar69}) the question of M. Artin is also relevant in this context and in this case this question has a positive answer.  Let us recall that a formal power series $f(x)\in\k\lb x_1, \ldots,x_n\rb$ is called \emph{algebraic} if it is algebraic over the ring of polynomials $\k[x_1,\ldots,x_n]$. The ring of algebraic power series is denoted by $\k\langle x_1,\ldots,x_n\rangle$.  Indeed after A. Gabrielov gave a negative answer to Artin's question, D. Popescu showed that it has a positive answer in the case the ring of convergent power series is replaced by the ring of algebraic power series:

\begin{thm}\label{Pop} \cite{Po} Let $\k$ be a field and  $F(x,y)$ be a vector of algebraic power series in two sets of variables $x$ and $y$. Assume given a formal power series solution $\wdh{y}(x)=(\wdh{y}_1(x), \ldots,\wdh y_m(x))$ vanishing at $0$,
$$F(x,\wdh{y}(x))=0.$$
Moreover let us assume that $\wdh{y}_j(x) \in\k\llbracket x_1,\ldots,x_{\s_j}\rrbracket  $, $1\leq j\leq m$, for some integers $\s_j$, $1\leq \s_j\leq n$.\\
Then for any $c\in\N$  there exists an algebraic power series solution $\wdt{y}(x)$ such that for all $j$, $\wdt{y}_j(x)\in \k\langle x_1,\ldots,x_{\s_j}\rangle$ and $\wdt{y}(x)-\wdh{y}(x)\in (x)^c$.
\end{thm}

Let us remark that if $F(x,y)$ is a vector of polynomials in $y$ with coefficients in $\k\langle x\rangle$ we may drop the condition that $\wdh{y}(x)$ vanishes at $0$ by replacing $F(x,y)$ (resp. $\wdh{y}(x)$) by $F(x,y+\wdh y(0))$ (resp. $\wdh{y}(x)-\wdh y(0)$).
This result has a large range of applications (see  \cite{FB}, \cite{Mir} or \cite{Sh} for some recent examples). Its proof relies on an idea of Kurke from 1972 and the Artin approximation property of rings of type $\k\lb x\rb\langle z\rangle$  based on the General N\'eron Desingularization Theorem which is quite involved (see \cite{Po} or \cite{Sw}). \\
The first goal of this paper is to provide a new and elementary proof of Theorem \ref{Pop}
for equations $F(x,y)=0$ which are linear in $y$ (see Theorem \ref{THM1}). This shows that Theorem \ref{Pop}
is really easier in the case $F(x,y)$ is linear in $y$. Let us mention that in the case where there is only one nest (i.e. when there is a given $k\leq n$ such that $\s_i=k$ or $n$ for every $i$) this has been already proven by E. Bierstone and P. Milman (see Theorem 12.6 \cite{BM}). In fact our proof is based on a reduction to this case.\\

In the second part of this paper we investigate the relationship between Artin's question and the following conjecture of A. Grothendieck (see \cite[p. 13-08]{Gro}):

\vspace{.1cm}

\emph{If $\phi:\C\{x\}/I\lgw \C\{y\}/J$ is an injective morphism of analytic algebras then the corresponding morphism $\wdh \phi:\C\lb x\rb/I\C\lb x\rb\lgw \C\lb y\rb/J\C\lb y\rb$ is again injective.}

In fact the counterexample of A. Gabrielov to Artin's question is built from a counterexample to the conjecture of A. Grothendieck he gave in \cite{Ga}. Even if it is obvious that the counterexample of Gabrielov to Grothendieck's conjecture provides a negative answer to the question of M. Artin, the relationship between these two problems is not clear in general.

The second goal of this paper is to clarify the relationship between Grothendieck's conjecture and Artin's question. We show in a general framework (i.e. not only for the rings of convergent power series or algebraic power series but for more general families of rings - cf. Definition \ref{adm}) that Grothendieck's conjecture  is equivalent to the question of M. Artin in the case where $F(x,y)$ is linear in $y$ (see Theorem \ref{Thm1}). Let us mention that it is well known that Grothendieck's conjecture is equivalent to  Artin's question  for some very particular  $F(x,y)$ which are linear in $y$ (see \cite{Be} and \cite{Ro}) but, to the best of our knowledge, it was not known that they are equivalent for all $F(x,y)$  linear in $y$.

We also prove (see Theorem \ref{Thm1}) that these two problems are equivalent to the problem of the commutation of two operations:  the operation of replacing an ideal by its completion and the operation of replacing an ideal by one of its elimination ideals (see \ref{sep}).

Finally we mention that the question of Grothendieck has been widely studied in the case of convergent power series rings and it has been shown that the answer is positive for some particular cases (see for instance \cite{A-vdP},  \cite{Ga2}, \cite{E-H}, \cite{Mi}, \cite{Iz} or \cite{To2}). One of  them, similar to our situation, is the case of a morphism $\phi: \k\{x\}/I\lgw \k\{y\}/J$ where the images of the $x_i$ are  algebraic power series and the ideals $I$ and $J$ are prime and generated by algebraic power series. For such morphisms it is shown that $\phi$ is injective if and only if $\wdh \phi$ is injective (it has been proven in several steps in \cite{To},  \cite{Be},  \cite{Mi} and \cite{Ro1}).

\section{Acknowledgements} This research was started in the frame of the Jean Morlet Chair {\sl Artin Approximation in Singularity Theory}, held at CIRM (Marseille, France) from January until June 2015. The Chair was held by Prof. Herwig Hauser and the Local Project Leader was the third author. We are  grateful for the hospitality and support of CIRM during the main stage of this work. We would like to thank Profs. H. Hauser and M.E. Alonso for their very useful comments and suggestions.  We also would like to thank E. Bierstone for having indicated to us the reference \cite{BM} where is proven a result similar to our Proposition \ref{p}.


\section{Linear Nested Artin Approximation for algebraic series}
We will prove the following linear version of Theorem \ref{Pop}:

\begin{thm}[Linear Nested Artin Approximation Theorem]\label{THM1}
Let $m,n,p$ be positive integers,
$T$ be a $p\times m$ matrix  with entries in $\k\lag x\rag:=\k\lag x_1, \ldots,x_n\rag$, 
$b=(b_1,\ldots,b_p) \in \k\lag x\rag^p$ and $\s :\{1, \ldots,m\}\lgw \{1, \ldots,n\}$ be a  map. Let $y=(y_1, \ldots,y_m)$ be a vector of new variables.
Then for any solution $\wdh y(x)$ in  $$ \k\lb  x_1, \ldots,x_{\s(1)}\rb \times \cdots\times\k\lb  x_1, \ldots,x_{\s(m)}\rb $$ of the following system of linear equations
\begin{equation}\label{eqq} Ty=b\end{equation}
and for any integer $c$ there exists a solution $y(x)$ in
$$\k\lag x_1, \ldots,x_{\s(1)}\rag\times \cdots\times\k\lag x_1, \ldots,x_{\s(m)}\rag$$
such that $y(x)-\wdh y(x)\in (x)^c\k\lb x\rb^m$.

\end{thm}

We begin by giving some intermediate results:

\begin{lem}\label{lem1}
Let $(A,\m)$ be a complete normal local domain, $x=(x_1,\ldots,x_n)$ and $y=(y_1,\ldots,y_m)$. Let $B=A\lb x\rb\lag y\rag$ be the algebraic closure of $A\lb x\rb[ y]$ in $A\lb x,y\rb$ and $f\in B$. Then there exist $g$ in the algebraic closure $A\lag y, z\rag$ of $A[y,z]$ in $A\lb y,z\rb$, with $z=(z_1,\ldots,z_s)$ for some $s\in\N$, and $\wdh z\in A\lb x\rb^s$ such that $f=g(y,\wdh z)$.
\end{lem}

\begin{proof}

By replacing $f$ by $f-f(0,y)$ we may assume that $f\in (x)B$. Note that $B$ is the Henselization of $C=A\lb x\rb[y]_{(\m,x,y)}$ by \cite[44.1]{Na} and so there exists some \'etale neighborhood of $C$ containing $f$. Using for example \cite[Theorem 2.5]{Sw} there exists a monic polynomial $F$ in $u$ over $A\lb x\rb[y]$ and $h\in (\m,x,y) A\lb x\rb \langle y\rangle$ such that $F(h)=0$, $(\partial F/\partial u)(h)\not \in (\m,x,y)$ and $f\in A\lb x\rb [y,h]_{(\m,x,y)\cap A\lb x\rb [y,h]}$, let us say $f=P(y,h)/Q(y,h)$ for some $P(y,u), Q(y,u)\in A\lb x\rb [y,u]$, $Q(y,h)\notin (\m,x,y)\cap A\lb x\rb [y,h]$.\\
Let us write
$$Q(y,h)=\sum_{\a,i}(q_{\a i}+\wdh w_{\a i})y^\a h^i$$
where $q_{\a i}\in A$ and $\wdh w_{\a i}\in (\m+(x))A\lb x\rb$  for every $\a$, $i$. We set
$$\tilde Q=\sum_{\a,i}(q_{\a i}+ w_{\a i})y^\a u^i$$
for new indeterminates $w_{\a i}$. Since $Q(y,h)\notin (\m,x,y)$, $\wdh w_{\a i} \in (\m+(x))A\lb x\rb$ and $h\in(\m,x,y) A\lb x\rb \langle y\rangle$, we have that $q_{00}$ is a unit of $A$. So $\tilde Q$ is invertible in $A\langle y,u,w_{\a,i}\rangle$. Moreover
$$\tilde Q^{-1}(y,h,\wdh w_{\a i})=Q(y,h)^{-1}$$ by uniqueness of the inverse. \\
Thus by adding the new $w_{\a i}$ and the coefficients of $P$ from $A\lb x\rb$ as new $\wdh w$, we see that our lemma works for $f$ as soon as it works for $h$.  So we can replace $f$ by $h$ and assume $f\in (\m,x,y) A\lb x\rb \langle y\rangle$, $F(f)=0$ and $F'(f):=(\partial F/\partial u)(f)\not \in (\m,x,y)$. Let us write $F=\sum_{\alpha,j}F_{\alpha j}y^{\alpha}u^j$ for some $F_{\alpha j}\in A\lb x\rb $.

Set ${\wdh z}_{\alpha i}=F_{\alpha i}-F_{\alpha i}(0)\in (x)A\lb x\rb $, ${\wdh z}=({\wdh z}_{\alpha i})$ and $G:=G(y,u,z)=\sum_{\alpha i}(F_{\alpha i}(0)+z_{\alpha i})y^{\alpha} u^i$ for some new variables $z=(z_{\alpha i})$. We  have $G(y,u,{\wdh z})=F$. Set $G'=\partial G/\partial u$. As
$$G(y,f,0)\equiv G(y,f,{\wdh z})\equiv F(f)\equiv 0\ \mbox{ modulo}\ (\m,x,y,u),$$ $$G'(y,f,0)\equiv G'(y,f,{\wdh z})\equiv F'(f)\not \equiv 0 \ \mbox{modulo}\ (\m,x,y,u)$$ we get $G(y,0,z)\equiv 0$, $G'(y,0,z)\not \equiv 0$ modulo $(\m,y,z)A\langle y,z\rangle$. By the Implicit Function Theorem there exists $g\in (\m,y,z) A\langle y,z\rangle$
such that $G(y,g,z)=0$. It follows that $G(y,g(y,{\wdh z}),z)=0$. But $F=G(y,u,{\wdh z})=0$ has just one solution $u=f$ in $(\m,x,y)B$ by the Implicit Function Theorem and so $f=g(y, {\wdh z})$.

\end{proof}


The following result can be rephrased as a particular case of Theorem 12.6 \cite{BM}, and the proof we give here, for the sake of completeness,  follows essentially the same principle as the proof given in \cite{BM}.

\begin{prop}\label{p}
We set $x=(x_1, \ldots,x_n)$ and $y=(y_1, \ldots,y_m)$ and let $M$ be a submodule of $\k\lag x,y\rag^p$. Then
$$\k\lb x\rb(M\cap\k\lag x\rag^p) = \wdh M\cap \k\lb x\rb^p$$
where $\wdh M=\k\lb x,y\rb M$ denotes the $(x,y)$-adic completion of $M$.
Moreover, if $c\in \mathbb{ N}$ and   ${\wdh u}=\sum_{i=1}^r {\wdh v}_i \o_i\in \wdh M\cap \k\lb x\rb^p$  for some $\o_i\in M$, ${\wdh v}_i\in \k\lb x,y\rb$ then there exist $v_{ic}\in \k\langle x,y\rangle$ such that  $ v_{ic}\equiv {\wdh v}_i\  \mbox{modulo} \ (x,y)^c\k\lb x,y\rb$,  $u_c=\sum_{i=1}^t  v_{ic}\o_i\in M\cap \k\lag x \rag^p$ and
 $\wdh u$ is the limit of  $(u_c)_c$ in the $(x)$-adic topology.
\end{prop}

\begin{proof}
Of course we always have $\k\lb x\rb(M\cap\k\lag x\rag^p) \subset  \wdh M\cap \k\lb x\rb^p$. So we only have to prove the opposite inclusion. \\
Let $\o_1,\ldots, \o_r$ be generators of $M$ and
 $\wdh u(x)$ be an element of $\wdh M\cap \k\lb x\rb^p$. Such an element $\wdh u(x)$ has the form
\begin{equation}\label{eq_mod}\wdh u(x)=\sum_{\ell=1}^r\wdh v_{\ell}(x,y)\o_{\ell}\end{equation}
for some formal power series $\wdh v_{\ell}(x,y)$. The components of Equation \eqref{eq_mod} provide a system of $p$ linear equations as follows:
\begin{equation}\label{lin_eq}T \wdh v(x,y)=\wdh u(x)\end{equation}
where $T$ is a $p\times r$ matrix with entries in $\k\lag x,y\rag$ and $\wdh v(x,y)$ is the vector of entries $\wdh v_{\ell}(x,y)$. \\
The morphism $\k\lb x \rb\lag y \rag \lgw \k\lb x ,y \rb$ being faithfully flat, for any integer $c$ there exists a solution $\wdt v (x ,y )\in \k\lb x \rb\lag y \rag^{r}$ of \eqref{lin_eq} such that
$$\wdt v (x ,y )-\wdh v (x ,y )\in (x ,y )^c\k\lb x,y \rb ^r.$$
Indeed, choose $v'(x,y)\in \k\lb x \rb\lag y\rag^r$ such that
$v'(x,y)- \wdh v(x,y)\in (x,y)^c\k\lb x,y \rb$. By faithfully flatness the linear system
$$\wdh u(x)=\sum_{\ell=1}^r v_{\ell}\o_{\ell},\ v'(x,y)-v=\sum_{|\alpha|+|\beta|=c} x^{\alpha}y^{\beta} w_{\alpha,\beta}$$ has  a
solution $\wdt v(x,y)$, $\wdt w(x,y)$ in $\k\lb x \rb\lag y\rag$ since it has one in $\k\lb x,y \rb $.

Thus from now on we may assume that $\wdh v (x ,y )\in  \k\lb x \rb\lag y \rag^{r}$. By Lemma \ref{lem1}  there exist a new set of variables $z=(z_1, \ldots,z_s)$, algebraic power series $g_\ell(y ,z)\in\k\lag y ,z\rag$ for $1\leq \ell \leq r$ and formal power series $\wdh z_1(x ),\ldots,\wdh z_s(x )\in(x )\k\lb x \rb$ such that
$$\wdh v _\ell(x ,y )=g_\ell(y ,\wdh z_1(x ), \ldots,\wdh z_s(x )).$$

Then, by replacing $v_\ell$ by $g_\ell(y,z)$ for $\ell=1,\ldots,r$ in the linear system of equations $T\cdot v=\wdh u(x)$  we obtain  a new system of (non linear) equations
$$f(x ,y ,\wdh u (x ),\wdh z(x ))=0$$ where $f(x ,y ,u ,z)$ is a vector of algebraic power series.

Let $\mathcal I$ denote the ideal of $\k\lag x ,u ,z\rag$ generated by all the coefficients of the monomials in $y$ in the expansion of the components of $f$ as power series in $(y_1, \ldots,y_m)$. Let $h_1,\ldots, h_t$ be a system of generators of $\mathcal I$. By assumption $(\wdh u (x ),\wdh z(x ))$ is a formal power series solution of the system
\begin{equation}\label{eqqq}h_1(x ,u ,z) = \cdots = h_t(x ,u ,z)=0.\end{equation}

Thus by Artin Approximation Theorem for algebraic power series \cite{Ar69}, for any integer $c\geq 0$  there exists $(\wdt u (x ),\wdt z(x ))\in\k\lag x \rag^{p+s}$ solution of the system \eqref{eqqq} with
$$\wdt u_\kappa(x )-\wdh u_\kappa(x )\in (x)^c\k\lb x\rb , \wdt z_k(x )-\wdh z_k(x )\in (x )^c\k\lb x\rb \ \ \forall \kappa, k.$$

Thus $(\wdt u(x),\wdt v(x,y))$ is a solution of the system \eqref{lin_eq} where
$$\wdt v_\ell(x ,y) = g_\ell(y ,\wdt z_1(x ), \ldots,\wdt z_s(x )) \ \ \forall \ell.$$
In particular  $\wdt u(x)\in M\cap \k\lag x\rag^p$ and $M\cap\k\lag x\rag^p$ is dense in $\wdh M\cap\k\lag x\rag$.\\
Moreover by Taylor's formula we have that
$$\wdt v_\ell(x,y )-\wdh v_\ell(x,y)\in (x,y )^c \k\lb x,y\rb\ \text{ for } 1\leq \ell \leq r.$$

\end{proof}

The next theorem is a key result to reduce the proof of Theorem \ref{THM1} to the case of only one nest (i.e. when there is  $k\leq n$ such that $\s_i=k$ or $n$ for every $i$). This one nest case is Proposition \ref{p}. As we have previously said, the linear one nest case was proven by E. Bierstone and P. Milman (see Theorem 12.6 \cite{BM}). \\

\begin{thm}\label{p1} Let $M\subset \k\langle x\rangle^p$ be a finitely generated $\k\langle x\rangle$-submodule and $\s :\{1, \ldots,p\}\lgw \{1, \ldots,n\}$ be a weakly increasing function.  Then
$$\NN=M\cap (\k\langle x_1,\ldots,x_{\s(1)}\rangle\times \ldots \times \k\langle x_1,\ldots,x_{\s(p)}\rangle)$$
is dense in
$$\NN'=(\k\lb x\rb M)\cap (\k\lb x_1,\ldots,x_{\s(1)}\rb\times \ldots \times \k\lb x_1,\ldots,x_{\s(p)}\rb).$$
 Moreover, if $c\in \mathbb {N}$ and   ${\wdh u}=\sum_{i=1}^t {\wdh v}_i \o_i\in {\NN'}$  for some $\o_i\in M$, ${\wdh v}_i\in \k\lb x\rb$ then there exist $v_{ic}\in \k\langle x\rangle$ such that  $ v_{ic}\equiv {\wdh v}_i\  \mbox{modulo} \ (x)^c\k\lb x\rb$,  $u_c=\sum_{i=1}^t  v_{ic}\o_i\in \NN$ and
 $\wdh u$ is the limit of  $(u_c)_c$ in the $(x)$-adic topology.
\end{thm}

\begin{proof}
Apply induction on $p$, the case $p=1$ being
done in Proposition \ref{p}.  Assume that $p>1$. We may reduce to the case when $\s(p)=n$ replacing $M$ by $M\cap \k\langle x_1,\ldots,x_{\s(p)}\rangle^p$ if $\s(p)<n$. Let
$$q:\k\lb x\rb^p\to \k\lb x\rb^{p-1}$$  be the projection on the first $p-1$ components and
$$q':\k\lb x\rb^p\to \k\lb x\rb$$
be the projection on the last component.
Let ${\wdh u}=({\wdh u_1},\ldots, {\wdh u}_p)\in {\NN'},$ and $M_1=q(M)$.
Assume that  ${\wdh u}=\sum_{i=1}^t  {\wdh v}_i\o_i$ for some ${\wdh v}_i\in \k\lb x\rb$, $\o_i\in M$.
By the induction hypothesis applied to $M_1$ and $q(\wdh u)$, for every  $c\in \mathbb{ N}$ there exists  $ v_{ic}\in \k\langle x\rangle$ with $v_{ic}\equiv {\wdh v}_i\ \mbox{modulo}\ (x)^c\k\lb x\rb$
 such that
 $$u'_c=\sum_{i=1}^t  v_{ic}q(\o_i)\in q(\NN)=M_1\cap (\k\langle x_1,\ldots,x_{\s(1)}\rangle\times \ldots \times \k\langle x_1,\ldots,x_{\s(p-1)}\rangle)$$
 and $q({\wdh u})$ is  the limit of  $(u'_c)_c$ in the $(x)$-adic topology.

Now,  let $u''_c=\sum_{i=1}^t v_{ic}q'(\o_i)\in \k\langle x_1,\ldots,x_n\rangle $. We have $u''_c\equiv q'({\wdh u})$ modulo $(x)^c\k\lb x\rb$. Then  $u_c=(u'_c,u''_c)=\sum_{i=1}^t v_{ic}\o_i\in \NN$ since $\s(p)=n$, $u_c\equiv {\wdh u}$ modulo $(x)^c\k\lb x\rb^{p}$  and $\wdh u$ is the limit of  $(u_c)_c$ in the $(x)$-adic topology.

\end{proof}

\begin{proof}[Proof of Theorem \ref{THM1}]
First of all we may assume that $\s$ is weakly increasing after permuting the $y_i$.\\
If $b=0$  then it is enough
to apply Theorem \ref{p1} for the module $M$ of the solutions of $Ty=0$ in $A=\k\langle x\rangle$. Suppose that $b\neq 0$. Replace the system $Ty=b$ by the homogeneous system of linear polynomials
$$T'y':=Ty-by_0=0$$
 from $A[y_0,y]^p$ where $y'=(y_0,y)$. A nested formal solution $\wdh y$ of $Ty=b$ in $\k\lb x\rb^{m}$ with  ${\wdh y}_i\in \k\lb x_1,\ldots,x_{\s(i)}\rb$, $1\leq i\leq m$ induces a nested formal solution $({\wdh y}_0,{\wdh y}) $, ${\wdh y}_0=1$ of $T'y'=0$ with $\s(0)=\s(1)$. As above, for all $c\in \N$ we get a nested algebraic   solution $(y_0(x),y(x))$ of $T'y'=0$ with $y_i(x)\in \k\langle x_1,\ldots,x_{\s(i)}\rangle$ and $y_i(x)\equiv {\wdh y}_i\ \mbox{modulo}\ (x)^c\k\lb x\rb$ for all $0\leq i\leq m$. In particular $y_{0}(0)=1\neq 0$ and $y_{0}(x)$ is a unit.
 Thus
$$(y_{0}(x)^{-1}y_{1}(x), \ldots,y_{0}(x)^{-1}y_{m}(x))$$
is an algebraic nested solution of $Ty-b=0$. Moreover, for all $j\geq 1$, we have:
$$y_{0}(x)^{-1}y_{j}(x)-\wdh y_j(x)=(y_{0}(x)^{-1}-1)y_{j}(x)+(y_{j}(x)-\wdh y_j(x))\in (x)^c.$$
\end{proof}



\section{Linear nested approximation property}

In the second part of this paper we generalize the method used to prove Theorem \ref{THM1} in order to show that the question of A. Grothendieck, for local subrings of the ring of formal power series,  is equivalent to the nested Artin approximation property for linear equations. We begin by giving several definitions.

\begin{defi}\label{adm}
Let $\k$ be a field. An \emph{admissible family of rings} is an increasing sequence of rings $\F=(R_n)_{n\in\N}$ satisfying  the following properties:
\begin{enumerate}
\item For every integer $n\geq 0$ the ring $R_n$ is a $\k$-subalgebra of $\k\lb x_1, \ldots,x_n\rb$ (in particular $R_0=\k$).
\item For every integer $n\geq 0$, $\k[x_1, \ldots,x_n]\subset R_n$.
\item For every integer $n> 0$ the ring $R_n$ is a Noetherian  local ring whose maximal ideal is generated by $x_1$,\ldots, $x_n$.
\item For every integer $n$ the completion of $R_n$ is $\k\lb x_1,\ldots,x_n\rb$.
\item For every integers $m,n$ with $0\leq m\leq n$ we have
$$R_n\cap \k\lb x_1, \ldots,x_m\rb=R_m.$$
\end{enumerate}
When an admissible family of rings is given, any element of a member of this family is called an \emph{admissible power series}.
\end{defi}

Sometimes we will emphasize the dependency of $R_n$ on the variables $(x_1,\ldots,x_n)$ by writing $R_n= \k\lg x_1, \ldots,x_n\rg$ for $n\in \N$.

\begin{example}
The following families of rings are admissible:
\begin{itemize}
\item The rings of convergent power series over a valued field $\k$.
\item The rings of algebraic power series over a field $\k$.
\item The rings of formal power series.
\item The rings of  germs of rational functions at  $0\in\k^n$, $\k[x_1, \ldots,x_n]_{(x_1, \ldots,x_n)}$.
\end{itemize}
\end{example}


\subsection{Krull topology}

Let $(A,\m)$ be a Noetherian local ring. The \emph{Krull topology} of $A$ is the topology in which the ideals $\m^c$ constitute a basis of neighborhoods of the zero of $A$. For a $A$-module $M$ the Krull topology of $M$  is the one in which the submodules $\m^cM$ constitute a basis of neighborhoods of the zero of $M$. The completion of $A$ (resp. $M$) for the Krull topology is denoted by $\wdh A$ (resp. $\wdh M)$. We have the following lemma asserting that the topological closure of a finite module and its completion coincide:

\begin{lem}\label{ZS}(\cite[Corollary 2, p. 257]{SZ})
If $N$ is a $A$-submodule of  a finite $A$-module $M$ then the closure of $N$ in $\wdh M$ is  $\wdh N=\wdh A N$.
\end{lem}

\begin{defi}
If $M$ is a $A$-module where $(A,\m)$ is a Noetherian local ring and $E$ is a subset of $M$, we say that an element $f\in M$ may \emph{be approximated by elements of $E$} if $f$ is in the closure (for the Krull topology) of $E$ in $M$, i.e. if for every integer $c$ there exists $f_c\in E$ such that $f-f_c\in \m^c M$.
\end{defi}


\subsection{Strong elimination property}\label{sep}
One says that an admissible family of rings $\F=(\k\lg x_1, \ldots,x_n\rg)_n$ has the \emph{strong elimination property for ideals}  if for every two sets of variables $x$ and $y$ and every ideal $I$ of $\k\lg x,y\rg$ we have
\begin{equation}\label{sei}(I\cap \k\lg x\rg)\k\lb x\rb=\wdh I\cap\k\lb x\rb\end{equation}
where $\wdh I$ denotes the ideal of $\k\lb x,y\rb$ generated by $I$.\\
One says that the admissible family $\F$ has the \emph{strong elimination property for modules}  if for every two sets of variables $x$ and $y$, every positive integer $p$ and every $\k\lg x,y\rg$-submodule $M$ of $\k\lg x,y\rg^p$ we have
\begin{equation}\label{sem}\k\lb x\rb(M\cap \k\lg x\rg^p)=\wdh M\cap\k\lb x\rb^p\end{equation}
where $\wdh M$ denotes the $\k\lb x,y\rb$-submodule of $\k\lb x,y\rb^p$ generated by $M$.\\

\begin{rem}\label{rmk_elim}Since $I\cap\k \lg x\rg\subset \wdh I\cap\k\lb x\rb$ (resp. $M\cap\k \lg x\rg^p\subset \wdh M\cap\k\lb x\rb^p$ , Lemma \ref{ZS} shows that \eqref{sei} (resp. \eqref{sem}) is equivalent to say that the elements of $\wdh I\cap\k\lb x\rb$ (resp. $\wdh M\cap\k\lb x\rb^p$) may be approximated by elements of $I\cap\k\lg x\rg$ (resp. $M\cap\k\lg x\rg^p$).
\end{rem}


\subsection{Linear nested approximation property}\label{lnap}
We say that an admissible family of rings $\F=(\k\lg x_1, \ldots,x_n\rg)_n$ has the \emph{linear nested approximation property} if the following property holds:\\
For every positive integers $m,n,p$,
every $p\times m$ matrix $T$   with entries in $\k\lg x\rg:=\k\lg x_1, \ldots,x_n\rg$, 
every $b=(b_1,\ldots,b_p) \in \k\lg x\rg^p$ and every map $\s :\{1, \ldots,m\}\lgw \{1, \ldots,n\}$ we have the following: let $y=(y_1, \ldots,y_m)$ be a vector of new variables.
Then the set of solutions $y(x)$ in  $$ \k\lg  x_1, \ldots,x_{\s(1)}\rg \times \cdots\times\k\lg  x_1, \ldots,x_{\s(m)}\rg $$ of the following system of linear equations
\begin{equation}\label{eq} Ty=b\end{equation}
is dense in the set of formal solutions in
$$\k\lb x_1, \ldots,x_{\s(1)}\rb\times \cdots\times\k\lb x_1, \ldots,x_{\s(m)}\rb.$$


\subsection{Strongly injective morphisms}\label{sim}
\begin{defi} \label{strongly-injective}
Let $\phi : A\lgw B$ be a morphism of local rings. We denote by $\wdh \phi$ the induced morphism $\wdh A\lgw \wdh B$. One says that $\phi$ is \emph{strongly injective} if $\wdh \phi$ is injective.
\end{defi}


\begin{defi}
We say that an admissible family of rings $\F=(\k\lg x_1, \ldots,x_n\rg)_n$ has the \emph{strong injectivity property} if for every integers $n$ and $m$ and every ideals $I$ of $\k\lg x_1, \ldots,x_n\rg$ and $J$ of $\k\lg y_1, \ldots,y_m\rg$,  every injective morphism of local rings
$$\frac{\k\lg x\rg}{I}\lgw \frac{\k\lg y\rg}{J}$$
is strongly injective.
\end{defi}

\begin{rem}
Definition \ref{strongly-injective} is not the classical one. In \cite{A-vdP} a morphism $\phi : A\lgw B$ is called strongly injective if $\wdh \phi(\wdh A)\cap B=\phi(A)$. This definition, which is the classical one, is stronger than the one we use in this paper. Nevertheless we will prove that if an admissible family of rings $(\k\lg x_1, \ldots,x_n\rg)_n$ has the strong injectivity property then for any morphism of local rings $\phi : A=\dfrac{\k\lg x\rg}{I}\lgw B=\dfrac{\k\lg y\rg}{J}$ we have $\wdh \phi(\wdh A)\cap B=\phi(A)$ (see Corollary \ref{cor}).
\end{rem}


The main result of this part is the following:

\begin{thm}\label{Thm1}
For an admissible family of rings $\F=(\k\lg x_1, \ldots,x_n\rg)_n$  the following properties are equivalent:
\begin{enumerate}
\item[(i)] $\F$ has the strong elimination property for ideals.
\item[(ii)] $\F$ has the strong elimination property for modules.
\item[(iii)] $\F$ has the linear nested approximation property.
\item[(iv)] $\F$ has the strong injectivity property.
\end{enumerate}
\end{thm}

\begin{cor}\label{cor}
Let $\F=(\k\lg x_1, \ldots,x_n\rg)_n$ be an admissible family having the strong injectivity property. Then for any morphism of local rings
$$\phi : A=\frac{\k\lg x\rg}{I}\lgw B=\frac{\k\lg y\rg}{J}$$
we have
$$\wdh \phi(\wdh A)\cap B= \phi(A).$$
In particular if $\wdh \varphi$ is surjective then $\varphi$ is surjective too.
\end{cor}

\begin{proof}
Clearly $\phi(A)\subset \wdh \phi(\wdh A)\cap B$. Let us prove the reverse inclusion.\\
We can replace $\phi$ by $\phi\circ\pi$ where $\pi : \k\lg x\rg\lgw \dfrac{\k\lg x\rg}{I}$ is the natural quotient morphism. This allows us to assume that $I=0$ and $A=\k\lg x\rg$. Let $\wdh f\in \k\lb x\rb$ such that $\wdh \phi(\wdh f)=b\in B$. Let us denote by
$\phi_i(y)$ an admissible power series of $\k\lg y\rg$ which is the image of $x_i$ by $\phi$ modulo $J$, for $i=1,\ldots,n$. Let  $q_1(y),\ldots, q_s(y)$ be generators of $J$. Thus, by assumption, there exist formal power series $\ \wdh{l}_j, \wdh{k}_i$, for $1\leq j\leq s$ and $1\leq i\leq n$,
such that
$$\wdh f(x)=b(y)+\sum_{j=1}^sq_j(y) \wdh l_j(x,y)+\sum_{i=1}^n(x_i-\phi_i(y)) \wdh k_i(x,y).$$
By the previous theorem the family of rings $\F$ has the linear nested approximation property, thus there exist admissible power series
$$f(x),  l_j(x,y), k_i(x,y)$$
such that
$$ f(x)=b(y)+\sum_{j=1}^sq_j(y)  l_j(x,y)+\sum_{i=1}^n(x_i-\phi_i(y))  k_i(x,y).$$
In particular, by replacing $x_i$ by $\phi_i(y)$ for all $i$ we see that
$\phi(f)=b$.
Thus $b\in \phi(\k\lg x\rg)$.
\end{proof}

\begin{rem}
Let $\F=(R_n)_n$ be an admissible family. Let $f\in R_n$ such that $f(0)=0$ and $\dfrac{\partial f}{\partial x_n}(0)\neq 0$. By the Implicit Function Theorem for formal power series there exists a unique formal power series $h(x')$ with $x'=(x_1, \ldots,x_{n-1})$ such that
$$f(x',h(x'))=0\text{ and } h(0)=0.$$
Thus, by Taylor's formula, there exists a formal power series $g(x)$ such that
$$f(x)+(x_n-h(x'))g(x)=0.$$
Since $\dfrac{\partial f}{\partial x_n}(0)\neq 0$ and $h(0)=0$ we have $g(0)\neq 0$, i.e. $g(x)$ is a unit. Hence we have, where $u(x)$ denotes the inverse of $g(x)$:
$$f(x)u(x)+x_n-h(x')=0.$$
Moreover, since $h(x')$ is unique, $u(x)$ is also unique and the linear equation
$$f(x)y_2+x_n-y_1=0$$
has a unique nested formal solution $(h(x'),u(x))$ whose first component vanishes at 0. Thus if the family $\F$ satisfies the equivalent properties of Theorem \ref{Thm1} then this family has to satisfy the Implicit Function Theorem (which is equivalent to say that the rings $R_n$ are Henselian local rings). \\
In particular the family of germs of rational functions at the origin of $\k^n$ does not satisfy the properties of Theorem \ref{Thm1}.\\
Since the ring of algebraic power series in $n$ variables is the Henselization of the ring of germs of rational functions at the origin of $\k^n$, this also shows that the family of algebraic power series is the smallest admissible family containing the family of germs of rational functions at the origin of $\k^n$ and satisfying the properties of Theorem \ref{Thm1} (by Theorem \ref{THM1}).
\end{rem}

\begin{rem}
Let $\F=(R_n)_n$ be an admissible family and $f$, $g$ two elements of $R_n$. Let us assume that $f$ is $x_n$-regular of order $d$, i.e. $f(0,x_n)=x_n^du(x_n)$ for some unit $u(x_n)$. By the Weierstrass division Theorem for formal power series there exists a unique vector
$$(q(x),a_0(x'),\ldots,a_{d-1}(x'))\in \k\lb x\rb\times \k\lb x'\rb^d$$
with $x'=(x_1,\ldots,x_{n-1})$ such that
$$g(x)=f(x)q(x)+\sum_{\kappa=0}^{d-1}a_\kappa(x')x_n^\kappa.$$
By the uniqueness of $(q(x),a_0(x'),\ldots,a_{d-1}(x'))$ if the family $\F$ has the linear nested approximation property then
$$(q(x),a_0(x'),\ldots,a_{d-1}(x'))\in R_n\times R_{n-1}^d.$$
Thus $\F$ satisfies the Weierstrass division Theorem if it satisfies the equivalent properties of Theorem \ref{Thm1}.\\
Let us mention that an admissible family of rings that satisfies the Weierstrass division Theorem is necessarily a family of Henselian local rings (see for instance \cite{D-L}). But it is still unknown if an admissible family of Henselian local rings satisfies the Weierstrass division Theorem (see for instance Remark 5.20 \cite{R}). See also \cite{La} for a partial result in this direction.
\end{rem}

\begin{rem}
The example of Gabrielov \cite{Ga} shows that the family of convergent power series over a characteristic zero valued field does not satisfy the properties of Theorem \ref{Thm1} (but this family satisfies the implicit function Theorem, it even satisfies the Weierstrass division Theorem). This example is the following one:

Let
 $$\phi : \C\{x_1,x_2,x_3\}\lgw\C\{y_1,y_2\}$$
  be the morphism of analytic $\C$-algebras defined by
  $$\phi(x_1)=y_1,\ \phi(x_2)=y_1y_2,\  \phi(x_3)=y_1e^{y_2}.$$
It  is not very difficult to show that $\phi$ and $\wdh\phi$ are both injective (see  \cite{Os}). Then A. Gabrielov remarked that there exists a formal but not convergent power series $\wdh g(x)$ whose image $h(y)$ by $\wdh \phi$ is convergent (see \cite{Ga}). This shows that Corollary \ref{cor} is not satisfied for convergent power series rings. Thus the properties of Theorem \ref{Thm1} are not satisfied in the case of convergent power series rings.

\end{rem}


\section{Proof of Theorem \ref{Thm1}}
We will prove the following implications:
$$(\text i)\Longrightarrow (\text{ii})\Longrightarrow ( \text{iii}) \Longrightarrow (\text iv)\Longrightarrow (\text i)$$

\subsection{Proof of (i) $\Longrightarrow$ (ii)}
In fact we will prove a stronger result that we will also use in the proof of (ii) $\Longrightarrow$ (iii). The proof of (i) $\Longrightarrow$ (ii) follows from the following lemma with $p=0$:
\begin{lem}\label{lem_inter}
Let $(\k\lg x_1,\ldots,x_n\rg)_n$ be an admissible family satisfying  the strong elimination property for ideals and let $M$ be a $\k\lg x,y\rg$-submodule of $\k\lg x,y\rg^{p+t}$.

Then
$M\cap(\{0\}^p\times\k\lg x\rg^t)$
is dense in
$\wdh M\cap (\{0\}^p\times\k\lb x\rb^t).$
\end{lem}

\begin{proof}

Let $S$ be the ring $\k\lg x,y,z,w\rg/(z,w)^2$ where $z=(z_1,\ldots,z_p)$ and $w=(w_1,\ldots,w_t)$ are new variables. Then the morphism of $\k\lg x,y\rg$-modules
$$\phi : \k\lg x,y\rg\times\k\lg x,y\rg^{p+t}\lgw S$$
$$(a,b_1,\ldots,b_p,c_1,\ldots,c_t)\lgm a+\sum_{i=1}^pb_iz_i+\sum_{j=1}^tc_jw_j$$
is a $\k\lg x,y\rg$-isomorphism. We denote by $\wdh \phi$ the isomorphism from $\k\lb x,y\rb^{p+t+1}$ to $\wdh S=\frac{\k\lb x,y,z,w\rb}{(z,w)^2}$ defined in the same way:
$$\wdh \phi(a,b_1,\ldots,b_p,c_1,\ldots,c_t)=a+\sum_{i=1}^pb_iz_i+\sum_{j=1}^tc_jw_j.$$
The image of $\{0\}\times M$ under $\phi$ is an ideal of $S$ denoted by $I$ and the image of $\{0\}\times\wdh M$ under $\wdh \phi$ is $\wdh I$. This is the idealization principle of Nagata.\\
 Moreover the image of $\{0\}\times (M\cap (\{0\}^p\times\k\lg x\rg^t))$ under $\phi$ is the ideal $I\cap \dfrac{\k\lg x,w\rg}{(w)^2}$ and the image of $\{0\}\times(\wdh M\cap (\{0\}^p\times\k\lb x\rb^t))$ is the ideal $\wdh I\cap \dfrac{\k\lb x,w\rb}{(w)^2}$.\\
By the strong elimination property for ideals $I\cap \dfrac{\k\lg x,w\rg}{(w)^2}$ is dense in $\wdh I\cap \dfrac{\k\lb x,w\rb}{(w)^2}$ hence $M\cap(\{0\}^p\times\k\lg x\rg^t)$ is dense in $\wdh M\cap (\{0\}^p\times\k\lb x\rb^t)$.

\end{proof}


\subsection{Proof of (ii) $\Longrightarrow$ (iii)}
We assume that $\F$ has the strong elimination property for modules and we fix a system of linear equations as \eqref{eq}. After a permutation of the $y_i$ we may assume that $\s$ is weakling increasing.\\
We call an \emph{admissible nested solution} (resp. \emph{formal nested solution}) of such a  system \eqref{eq} a solution in
$$ \k\lg  x_1, \ldots,x_{\s(1)}\rg \times \cdots\times\k\lg  x_1, \ldots,x_{\s(m)}\rg $$
$$(\text{resp. } \k\lb x_1, \ldots,x_{\s(1)}\rb\times \cdots\times\k\lb x_1, \ldots,x_{\s(m)}\rb).$$
We will show that the set of admissible nested solutions is dense, for the $\m$-adic topology, in the set of formal nested solutions.\\

$\bullet$ First we claim that we can assume that $b=0$, i.e. the system \eqref{eq} of linear equations is homogeneous. Indeed let us assume that the set of admissible nested solutions of any linear homogeneous system is dense in the set of formal nested solutions and let us fix a linear (non-homogenous) system as \eqref{eq}. Let $y(x)\in\k\lb x\rb^m$ be a formal nested solution of the system \eqref{eq}: $Ty=b$.\\
Let us write $a_{i,j}$ the entries of the $p\times m$ matrix $T$ and denote by $T'$ the matrix

$$T'=\left[\begin{array}{c} -b\, |\, T \end{array}\right]$$
and set $y'=(y_0,y_1, \ldots,y_m)$.

Let us extend  the previous function $\s$ to $\{0, \ldots,m\}$ by $\s(0)=\s(1)$.
Since $y(x)$ is a formal nested solution of \eqref{eq},  $y'(x)=(1,y(x))$ is a formal nested solution of the following linear homogeneous system:
\begin{equation}\label{eq'} T'y'=0\end{equation}
By assumption, for any given  integer $c\geq 1$, there exists an admissible nested solution $y'_c(x)=(y_{0,c}(x),y_{1,c}(x), \ldots,y_{m,c}(x))$ of \eqref{eq'} such that
$$y_{0,c}(x)-1\in (x)^c \text{ and } y_{j,c}(x)-y_j(x)\in (x)^c\ \ \ \forall j\geq 1.$$
In particular $y_{0,c}(0)=1\neq 0$ and $y_{0,c}(x)$ is a unit. Thus
$$(y_{0,c}(x)^{-1}y_{1,c}(x), \ldots,y_{0,c}(x)^{-1}y_{m,c}(x))$$
is an admissible nested solution of \eqref{eq}. Moreover, for all $j\geq 1$, we have:
$$y_{0,c}(x)^{-1}y_{j,c}(x)-y_j(x)=(y_{0,c}(x)^{-1}-1)y_{j,c}(x)+(y_{j,c}(x)-y_j(x))\in (x)^c.$$
Thus the set of admissible nested solutions of \eqref{eq} is dense in the set of formal nested solutions of \eqref{eq} and the claim is proven.
\\

$\bullet$ Let us consider a homogeneous linear system \eqref{eq} where $b=0$. The set of (non-nested) admissible solutions of such a system is a $\k\lg  x\rg $-submodule of $\k\lg  x\rg  ^m$ denoted by $M$. By Noetherianity this module is finitely generated. The set of (non-nested) formal solutions is the completion of $M$ denoted by $\wdh M$ (by flatness of $\k\lg x\rg\lgw \k\lb x\rb$ since $\k\lg x\rg$ is a Noetherian local ring). Thus  the following lemma shows that the nested admissible solutions are dense in the set of nested formal solutions and $\F$ has the linear nested approximation property:\\
\begin{lem}\label{elim_p} Let us assume that $\F$ has the strong elimination property for modules and  $M$ be a finite submodule of $\k\lg  x\rg ^m$. Then
$$M\cap\left(\k\lg  x_1, \ldots, x_{\s(1)}\rg \times \ldots\times \k\lg  x_1, \ldots,x_{\s(m)}\rg \right)$$
is dense in
$$\wdh M\cap\left(\k\lb x_1, \ldots, x_{\s(1)}\rb\times \ldots\times \k\lb x_1, \ldots,x_{\s(m)}\rb\right).$$
\end{lem}

To prove this lemma we proceed in a similar way as for Theorem \ref{p1}. But before we need to state a  preliminary result since the strong elimination property is apparently weaker than the condition used in Proposition \ref{p}. This  statement is the following lemma which is an analogue of Chevalley's Lemma for summands of modules (classical Chevalley's Lemma concerns decreasing sequences of ideals in complete local rings - see Lemma 7 \cite{Ch}):

\begin{lem}[Chevalley's Lemma]\label{chevalley}
Let $M$ be a $\k\lb x,y\rb$-submodule of $\k\lb x,y\rb^{p+t}$. Then  there exists a  function $\b :\N\lgw \N$ such that
$$M\cap ((x)^{\b(c)}\k\lb x\rb^p\times\k\lb x\rb^t))\subset M\cap(\{0\}^p\times\k\lb x\rb^t)+(x)^{c}\k\lb x\rb^{p+t} \ \ \forall c\in\N.$$
\end{lem}

\begin{proof}
For simplicity let us set $N:=M\cap(\{0\}^p\times\k\lb x\rb^t)$. Let us assume that there is an integer $c_0\in\N$ such that
$$M\cap ((x)^{\b}\k\lb x\rb^p\times\k\lb x\rb^t))\not\subset N+ (x)^{c_0}\k\lb x\rb^{p+t} \ \ \forall \b\in\N.$$
 So we have that
$$M\cap ((x)^{\b}\k\lb x\rb^p\times\k\lb x\rb^t)\not\subset N+(x)^{c}\k\lb x\rb^{p+t} \ \ \forall \b\in\N, \ \forall c\geq c_0.$$
By replacing $M$ by $M\cap\k\lb x\rb^{p+t}$ we may assume that $M$ is a submodule of $\k\lb x\rb^{p+t}$.
The module $M/(M\cap (x)^c\k\lb x\rb ^{p+t})$ is an Artinian module thus there is an integer $a(c)$ such that for all $\b\geq a(c)$:

$$M\cap ((x)^{a(c)}\k\lb x\rb^p\times\k\lb x\rb^t)+M\cap (x)^c\k\lb x\rb ^{p+t}=M\cap ((x)^{\b}\k\lb x\rb^p\times\k\lb x\rb^t))+M\cap (x)^c\k\lb x\rb ^{p+t}.$$
We may assume that $a(c)<a(c+1)$ for every $c$. Since
$$M\cap ((x)^{a(c)}\k\lb x\rb^p\times\k\lb x\rb^t)\subset M\cap ((x)^{a(c)}\k\lb x\rb^p\times\k\lb x\rb^t)+M\cap (x)^c\k\lb x\rb ^{p+t}$$
$$=M\cap ((x)^{a(c+1)}\k\lb x\rb^p\times\k\lb x\rb^t)+M\cap (x)^c\k\lb x\rb ^{p+t}$$
for a given $u_c\in M\cap ((x)^{a(c)}\k\lb x\rb^p\times\k\lb x\rb^t)$ there exists an element
$$u_{c+1}\in M\cap ((x)^{a(c+1)}\k\lb x\rb^p\times\k\lb x\rb^t)$$
 such that
$$u_c-u_{c+1}\in M\cap (x)^c\k\lb x\rb ^{p+t}.$$
Thus by choosing $u_{c_0}\in M\cap ((x)^{a(c_0)}\k\lb x\rb^p\times\k\lb x\rb^t)\backslash \left(N+M\cap (x)^{c_0}\right)$ we may construct a sequence $(u_c)_c$ as above. Hence this sequence is a Cauchy sequence and has a limit $u\in M$ since $M$ is a complete module. But $u_{c'}\in M\cap ((x)^{a(c)}\k\lb x\rb^p\times\k\lb x\rb^t)$ for every $c'\geq c$ and   $M\cap ((x)^{a(c)}\k\lb x\rb^p\times\k\lb x\rb^t)$ is a complete module thus
$$u\in \bigcap_{c\geq c_0} M\cap ((x)^{a(c)}\k\lb x\rb^p\times\k\lb x\rb^t)=M\cap(\{0\}^p\times\k\lb x\rb^t)=N$$
by Nakayama's Lemma.\\
On the other hand we have
$$u-u_{c_0}\in M\cap (x)^{c_0}\k\lb x\rb ^{p+t}$$
so $u_{c_0}\in N+M\cap (x)^{c_0}\k\lb x\rb ^{p+t}$ which contradicts the assumption on $u_{c_0}$.
\end{proof}

\begin{proof}[Proof of Lemma \ref{elim_p}]
We prove the lemma by induction on $m$, the case $m=1$ being equivalent to the strong elimination property for modules.
  Assume that $m>1$. By the strong elimination property for modules we may reduce to the case when $\s(m)=n$ by replacing $M$ by $M\cap \k\lg x_1,\ldots,x_{\s(m)}\rg^m$ if $\s(m)<n$. Let
$$q:\k\lb x\rb^m\to \k\lb x\rb^{m-1}$$
be the projection on the first $m-1$ components. Let
$${\wdh u}\in \wdh M\cap\left(\k\lb x_1, \ldots, x_{\s(1)}\rb\times \ldots\times \k\lb x_1, \ldots,x_{\s(m)}\rb\right)$$ and set $M_1=q(M)$.
By induction hypothesis applied to $M_1$ and $q(\wdh u)$, for every $c\in \N$ there exists   $u'_c\in M\cap\left(\k\lg  x_1, \ldots, x_{\s(1)}\rg \times \ldots\times \k\lg  x_1, \ldots,x_{\s(m)}\rg \right)$  such that  $q({\wdh u})-q(u'_c)\in (x)^c$.\\
Now $\wdh u-u'_c\in\left( (x)^c\k\lb x\rb^{m-1}\times\k\lb x\rb\right)\cap \wdh M$. Thus by Lemmas \ref{chevalley} and \ref{lem_inter} there exists a function $\b$ such that
$$\wdh u-u'_{\b(c)}\in (x)^c\k\lb x\rb^m\cap \wdh M+\overline{(\{0\}^{m-1}\times\k\lg x\rg)\cap M}$$
where $\overline{(\{0\}^{m-1}\times\k\lg x\rg)\cap M}$ denotes the closure of $(\{0\}^{m-1}\times\k\lg x\rg)\cap M$.
 Thus there exists $u''_c\in(\{0\}^{m-1}\times\k\lg x\rg)\cap M$ such that
 $$\wdh u-(u'_{\b(c)}+u''_c)\in (x)^c\k\lb x\rb^m$$
 and $$u'_{\b(c)}+u''_c\in M\cap\left(\k\lg  x_1, \ldots, x_{\s(1)}\rg \times \ldots\times \k\lg  x_1, \ldots,x_{\s(m)}\rg \right)$$
 since $\s(m)=n$.
\end{proof}

\subsection{Proof of (iii) $\Longrightarrow$ (iv)}

Let $$\phi :\frac{\k\lg x\rg}{I}\lgw \frac{\k\lg y\rg}{J}$$
be an injective morphism of local rings and let $\wdh f\in\Ker(\wdh \phi)$. The morphism $\phi$ is defined by  admissible power series $\phi_1(y),\ldots,\phi_n(y)$ such that
$$g(\phi_1(y), \ldots,\phi_n(y))\in J\ \ \ \forall g\in I$$
and,  for any power series $g$, the image of $g$ modulo $I$ is equal to
$$g(\phi_1(y), \ldots,\phi_n(y))\text{ modulo } J.$$
We still denote by $\wdh f$ a lifting of $\wdh f$ in $\k\lb x\rb$. Thus
$$\wdh f(\phi_1(y), \ldots,\phi_n(y))\in \wdh  J,$$
i.e. there exist formal power series $\wdh h_1(y),\ldots,\wdh h_s(y)$ such that
$$\wdh f(\phi_1(y), \ldots,\phi_n(y))=\sum_{j=1}^sq_j(y)\wdh h_j(y)$$
where the $q_j(y)$ are generators of the ideal  $J$. By Taylor's formula there exist formal power series $\wdh k_i(x,y)$ such that
\begin{equation}\label{nested_eq}\wdh f(x)-\sum_{j=1}^sq_j(y)\wdh h_j(y)=\sum_{i=1}^n(x_i-\phi_i(y))\wdh k_i(x,y).\end{equation}

By the linear nested approximation property, for any integer $c$, there exists a vector of  admissible power series
$$(f_c(x), h_{1,c}(x,y), \ldots,h_{s,c}(x,y),k_{1,c}(x,y), \ldots, k_{n,c}(x,y))$$
such that
$$f_c(x)-\sum_{j=1}^s q_j(y) h_{j,c}(x,y)=\sum_{i=1}^n(x_i-\phi_i(y)) k_{i,c}(x,y)$$
and
$$f_c(x)-\wdh f(x)\in (x)^c, \ h_{j,c}(x,y)-\wdh h_{j}(y)\in (x,y)^c,\ k_{i,c}(x,y)- \wdh k_{i}(x,y)\in (x,y)^c$$
for all $j$ and $i$.
By replacing $x_i$ by $\phi_i(y)$ for $i=1,\ldots,n$, we see that $\phi(f_c(x))=0$, thus $f_c(x)=0$ since $\phi$ is injective. Thus $\wdh f(x)\in (x)^c$ for all $c\geq 0$ thus $\wdh f(x)=0$ by Nakayama's Lemma. This shows that  $\phi$ is strongly injective.

\subsection{Proof of (iv) $\Longrightarrow$ (i)}
Let $I$ be an ideal of $\k\lg  x,y\rg $. Let $\phi$ be the following injective morphism induced by the inclusion $\k\lg x\rg\lgw \k\lg x,y\rg$:
$$\frac{\k\lg  x\rg }{I\cap\k\lg  x\rg }\lgw \frac{\k\lg  x,y\rg }{I}.$$
Then $(I\cap\k\lg  x\rg )\k\lb x\rb=\wdh I\cap\k\lb x\rb$ if and only if $\phi$ is strongly injective since
$$\Ker(\wdh \phi)=\frac{\wdh I\cap \k\lb x\rb}{(I\cap\k\lg x\rg)\k\lb x\rb}.$$




\end{document}